\documentclass[12pt]{amsart}

\usepackage[usenames,dvipsnames]{xcolor}

\usepackage{fancyhdr}
\usepackage{amsmath,amsthm,amsfonts,graphicx,amssymb,amscd, mathtools}
\usepackage[all,cmtip]{xy}
\usepackage{graphicx, overpic, float}
\usepackage[mathscr]{euscript}
\usepackage{hyperref}
\usepackage{enumerate}
\usepackage{microtype} 

\newtheorem{thm}{Theorem}[section]
\newtheorem*{jthm}{Theorem}

\newtheorem{cor}[thm]{Corollary}
\newtheorem{lem}[thm]{Lemma}

\newtheorem{prop}[thm]{Proposition}

\theoremstyle{definition}
\newtheorem{defn}[thm]{Definition}
\newtheorem{definition}[thm]{Definition}

\newcommand{\tec}{Teichm\"uller }

\usepackage{fancybox}

\usepackage{ulem}

\renewcommand{\emph}[1]{{\it #1}}

\usepackage{color}

\title[Meromorphic quadratic differentials and spiralling foliations]{Meromorphic quadratic differentials with complex residues and spiralling foliations}

\author{Subhojoy Gupta}
\author{Michael Wolf}

\address{Department of Mathematics, Indian Institute of Science, Bangalore 560012, India.} 

\email{subhojoy@math.iisc.ernet.in}

\address{Department of Mathematics, Rice University, Houston, Texas, 77005-1892, USA.}
\email{mwolf@rice.edu}

\date{\today}

\begin{document}
\setcounter{tocdepth}{4}
\maketitle

\begin{abstract}
A meromorphic quadratic differential with poles of order two, on a compact Riemann surface, induces a  measured foliation on the surface, with a spiralling structure at any pole that is determined by the complex residue of the differential at the pole.  We introduce the space of such measured foliations, and prove that for a fixed Riemann surface,  any such foliation is realized by a quadratic differential with second order poles at marked points. Furthermore, such a differential is uniquely determined if one prescribes complex residues at the poles  that are compatible with the transverse measures around them. This generalizes a theorem of Hubbard and Masur concerning  holomorphic quadratic differentials on closed surfaces, as well as a theorem of Strebel for the case when the foliation has only closed leaves. The proof involves taking a compact exhaustion of the surface, and considering a sequence of equivariant harmonic maps to real trees that do not have a uniform bound on total energy. 

\end{abstract}

\section{Introduction}

A holomorphic quadratic differential on a Riemann surface induces a singular foliation  equipped with a transverse measure. 
For a closed Riemann surface $\Sigma$, the Hubbard-Masur theorem (\cite{HubbMas}) asserts that this correspondence provides a homeomorphism between the space of holomorphic quadratic differentials and the space of equivalence classes of measured foliations on $\Sigma$. Both provide useful descriptions of the space of infinitesimal directions in Teichm\"{u}ller space at the point $\Sigma$, and the theorem established an important bridge between the complex-analytic and topological aspects of the  theory. The purpose of this article is to provide a generalization
of this result  to quadratic differentials with poles of order two.

For this, we shall introduce a space $\mathcal{MF}_2$ of  the corresponding topological objects on the surface, namely equivalence classes of measured foliations with \textit{centers} (disks foliated by closed circles or spirals - see Figure 3) around finitely many marked points. The data for such a foliation includes the transverse measure of a loop linking each distinguished point, in addition to the transverse measures of a system of arcs  on the surface that provide coordinates on the usual space of measured foliations introduced by Thurston  (see \cite{FLP}). 

On the complex-analytic side, at a pole $p$ of order two, a quadratic differential  is locally of the form :
\begin{equation}\label{ord2}
\left(\frac{a^2}{z^2} \right)dz^2\text{   where  } a\in \mathbb{C}^\ast 
\end{equation}
with respect to a  choice of coordinate chart around the pole.  

\begin{defn}[Residue]\label{defn:res} The constant $a$ in (\ref{ord2}) shall be called the \textit{(complex) residue} of the differential at the pole.  The ambiguity of sign of the square root is resolved by choosing the one that makes the real part  of $a$ positive, or in case it is zero,  the sign such that the imaginary part of $a$ is positive.  We shall denote the residue at a pole $p$ by $\text{Res}(p)$.
\end{defn} 

 The residue is in fact a coordinate-independent quantity. We also note that the coefficient $a^2$ is  often called the residue (or \textit{quadratic} residue) in the literature; the present definition makes subsequent notation less cumbersome.
 For example, the transverse measure of a loop around the pole, for the induced measured foliation, is a constant multiple of the real part of the residue as defined above (see (\ref{trans}) in  \S2.2).  \\

We shall prove: 

\begin{thm}\label{main} Let $\Sigma$ be a closed Riemann surface of genus $g$, with marked points $P=\{p_j\}_{1\leq j\leq n}$, such that $2g-2 + n>0$. Let $F\in \mathcal{MF}_2(\Sigma, P)$ be a measured foliation on $
\Sigma$ with centers at points of $P$. 
Then for any tuple  $(c_1, c_2, \ldots c_n) \in \mathbb{R}^n_+$ there is a unique meromorphic quadratic differential with 
\begin{itemize}
\item a pole of order two at each $p_j$,
\item an induced measured foliation that is equivalent to $F$, and 
\item a residue at $p_j$ having a prescribed imaginary part, namely,
\[
 \Im ( \text{Res}(p_j) ) =
\left\{
\!
\begin{aligned}
 c_j & \text{  if  } \Re ( \text{Res}(p_j) ) = 0\\
\ln c_j   & \text{  if  } \Re ( \text{Res}(p_j) ) > 0\\
\end{aligned}
\right.
\]

for each $1\leq j\leq n$.
\end{itemize}
\end{thm}

\textit{Remarks.} (i)  In the case when the real part of the residue is non-zero (and hence positive, by our sign convention), a neighborhood of the pole is foliated by leaves that spiral into it. The spiralling nature is \textit{not} captured by the parameters (namely, transverse measures of loops) determining the measured foliation $F$; indeed, different spiralling behaviour at the pole are isotopic to each other (see \S4). It is only the total transverse measure around the pole which is fixed by the choice of $F$, and this determines the real part of the residue. The free parameter is therefore the imaginary part of the residue; this then determines the spiralling behaviour, and its  sign represents the ``handedness" of the spiralling, that is, clockwise or counter-clockwise with respect to the orientation on the surface.  \\


(ii) The condition of prescribing the imaginary part of the residue can be replaced with the geometric requirement of
 \begin{enumerate}
  \item  prescribing the circumference of the cylinder surrounding the pole $p_j$  in the induced singular-flat metric (see \S2.2)  and
   \item prescribing  the ``handedness" of the spiralling mentioned in (i) above. 
   \end{enumerate}
   Note that the circumference in (1) must not be less than the total transverse measure around the pole.\\

(iii)  A special case is when the foliation $F$  has  the global property that all leaves are closed (excepting a connected critical graph) and foliate punctured disks around each pole. In this case, the transverse measure of any loop around a pole vanishes, since each linking loop can be chosen to be a closed leaf of $F$  (which has no transverse intersection with the foliation). From the relation between the residue and transverse measure, it follows that the real part of the residue is zero. Theorem \ref{main} then asserts that an additional positive real parameter (the imaginary part of the residue) uniquely determines a differential realizing $F$. This recovers a well-known  theorem of Strebel (see Theorem 23.5 of  \cite{Streb} or \cite{Liu}).\\

 The present paper deals with the case of poles of order \textit{exactly} two. The case of  order-$1$ poles, that we elide,  falls under the purview of classical theory: the corresponding quadratic differentials are then integrable, with the induced foliation having a ``fold" at the puncture.  The Hubbard-Masur theorem  can be applied after taking a double cover of the surface  branched at such poles. \\

The proof of Theorem \ref{main} involves the technique of harmonic maps to $\mathbb{R}$-trees used by  one of us (\cite{Wolf2}) in an alternative  proof of the Hubbard-Masur theorem. Such an $\mathbb{R}$-tree forms the leaf-space of the measured foliation when lifted to the universal cover $\widetilde{\Sigma}$, and the  required quadratic differential is then recovered as (a multiple of) the Hopf differential of the equivariant harmonic map from $\widetilde{\Sigma}$ to the tree. 
 The construction of the $\mathbb{R}$-tree as the leaf-space  ignores the differences in foliations caused by Whitehead moves; thus the analysis avoids issues relating to the presence of higher order zeroes of the differential, and their possible decay under deformation to multiple lower order zeroes, that was a difficulty in the approach of Hubbard and Masur. 

In the present paper the main difficulty lies in the fact that  surfaces have punctures and the corresponding harmonic maps  have infinite energy. The strategy for the proof of existence is to establish uniform energy bounds for harmonic maps defined on a compact exhaustion of the surface. 
 In the conclusion of the proof, some additional work is needed to verify that the Hopf differential of the limiting map has the desired complex residue. 

A similar strategy was used in (\cite{GW1}) for a generalization of Strebel's theorem (mentioned above)  to the case of quadratic differentials with poles of higher order.   There,  the global ``half-plane structure" of the foliation allowed us to consider harmonic maps from the punctured Riemann surface $(\Sigma \setminus P)$ to a $k$-pronged tree.  
In this paper we consider arbitrary foliations, which necessitates working in the universal cover and considering a target space that is a more general $\mathbb{R}$-tree. Furthermore, a difficulty unique to the case of double order poles is the ``spiralling" nature of the foliations at a (generic) such pole.

There are two novel technical elements in meeting the difficulties of the infinite energy in this setting of harmonic maps for such spiralling foliations. First, we use a version of the length-area method that is adapted to the collapsing maps for such foliations; the resulting estimate (Proposition~\ref{minim}) takes into account  both horizontal and vertical stretches.
Second, we need to control the supporting annulus of the spirals or twists for a map in the approximating sequence: this is done (Lemma~\ref{transv}) by using transverse measures (that determine distances in the target $\mathbb{R}$-tree) to bound the size of the subsurface to which the twisting can retreat.

Harmonic maps provides a convenient analytical tool in this paper; apart from the advantage 
over the approach of Hubbard-Masur in avoiding issues related to Whitehead moves mentioned earlier, a crucial fact used in the proof of the  uniform energy bound  (Lemma \ref{ebd}) is that for a non-positively curved target, such a map is unique and energy-minimizing.

In the forthcoming work \cite{GW2} we use ideas in these papers to prove a full generalization of the Hubbard-Masur theorem to quadratic differentials with higher-order poles.\\

\textbf{Acknowledgements.} The authors thank the anonymous referee for very helpful comments, and gratefully acknowledge support for collaborative travel by NSF grants DMS-1107452, 1107263, 1107367 ``RNMS: GEometric structures And Representation varieties" (the GEAR network). The first author was supported by the Danish Research Foundation Center of Excellence grant and the Centre for Quantum Geometry of Moduli Spaces (QGM), as well as the hospitality of MSRI (Berkeley).  He would like to thank Daniele Allessandrini and Maxime Fortier-Bourque for helpful discussions. The second author gratefully acknowledges support from the U.S.~National Science Foundation (NSF) through grants DMS-1007383 and DMS-1564374.

\section{Preliminaries}

\subsection{Quadratic differentials, measured foliations}

Let $\Sigma$ be a closed Riemann surface of genus $g\geq 2$.

\begin{defn} A \textit{meromorphic quadratic differential} on $\Sigma$  is a $(2,0)$-tensor that is locally of the form $q(z)dz^2$ where $q(z)$ is meromorphic.  Alternatively, it is a meromorphic section of $K^{\otimes 2}$, where $K$ is the canonical line bundle on $\Sigma$. 
\end{defn}

By Riemann-Roch, the complex vector space of meromorphic quadratic differentials with poles at points $p_1,p_2,\ldots p_n$ of orders at most $k_1,k_2,\ldots k_n$ respectively, has (complex) dimension $3g-3 + \sum\limits_{i=1}^n k_i$.\\

\begin{defn}\label{hf} The \textit{horizontal foliation} induced by a quadratic differential $q$ is a (singular) foliation on the surface whose leaves are the integral curves of the line field along which the quadratic differential is real and positive (that is, $\pm v\in T\Sigma$ such that $q(v \otimes v) \in \mathbb{R}_+$). The singularities are at the zeros of $q$, where there is a ``prong"-type singularity (see Figure 1), and at the poles (for poles of order two, as in this paper, see Figure 3).

\end{defn}

\begin{defn} The \textit{singular-flat metric} induced by a quadratic differential $q$ is a singular metric $ds^2_q$ which is locally of the form $ds^2_q = \lvert q(z)\rvert \lvert dz^2 \rvert$.  The singularities of the metric are at the zeros 
and first order poles of $q$, where there is a cone-angle $\pi(k+2)$ where $k$ is the order of the zero or pole. The metric is complete in the neighborhood of a second order pole; we can choose such a neighborhood to be isometric to a portion of a half-infinite cylinder (see \S2.2). 
\end{defn}

The \textit{horizontal} and \textit{vertical} lengths of an arc $\alpha$ are defined to be 
 the absolute values of the real and imaginary parts, respectively:
 \begin{equation}\label{meas}
 l_{h}(\alpha) =  \lvert \Re  \displaystyle\int\limits_\alpha \sqrt{ q(z)} dz \rvert,\text{ } l_{v}(\alpha) =  \lvert \Im \displaystyle\int\limits_\alpha \sqrt{ q(z) }dz\rvert
 \end{equation}
 where we note that the length of  a geodesic arc $\alpha$, which avoids singularities, in the induced metric is $l(\alpha) = \left(l_h(\alpha)^2 + l_v(\alpha)^2\right)^{1/2}$.\\

 \begin{figure}
  \centering
  \includegraphics[scale=1]{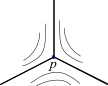}\\
 \caption{The induced (horizontal) foliation at a simple zero has a $3$-prong singularity with a cone angle of $3\pi$.}
\end{figure}

The horizontal foliation $\mathcal{F}(q)$ induced by a quadratic differential $q$  (Defn. \ref{hf})  can be equipped with a \textit{transverse measure}, namely the vertical length of any  transverse arc. Thus, the foliation $\mathcal{F}$ is \textit{measured} in the following sense:

\begin{defn}[Measured foliation]\label{mf} A measured foliation on a smooth surface $S$ of genus $g\geq 2$ is a smooth foliation with finitely many ``prong" singularities, equipped with a measure on transverse arcs that is invariant under transverse homotopy.  A pair of measured foliations are said to be \textit{equivalent}, if they differ by an isotopy of leaves and Whitehead moves (see Figure 2).

The space $\mathcal{MF}$ of equivalence classes of measured foliations  is homeomorphic to $\mathbb{R}^{6g-6}$. (See for example,  Expos\'{e} 6 of \cite{FLP}.)

One way to obtain this parameterization is to  consider a pants decomposition of the surface, and use the transverse measures of each of the $3g-3$ pants curves, together with the \textit{twisting number} of each. Here, the twisting number of a curve can be defined to be the nonnegative real number that is the transverse measure across a maximal annulus on the surface with a core curve homotopic to the given one. This is known as the Dehn-Thurston parametrization (see \cite{DehnThur} and \cite{PenHar} for details).

The same definition of a measured foliation as above holds for a surface with boundary; the leaves of the foliation may be parallel or transverse at the boundary components, and prong singularities may also lie on the boundary.
\end{defn}

\textit{Notation.} In this paper the \textit{induced measured foliation} for a (meromorphic) quadratic differential $q$ shall mean its horizontal foliation, denoted $\mathcal{F}(q)$. \\

As mentioned in the introduction, this article is motivated by the following correspondence between measured foliations and \textit{holomorphic} quadratic differentials.

\begin{jthm}[Hubbard-Masur \cite{HubbMas}]  Given a closed Riemann surface $\Sigma$ of genus $g\geq 2$, and a measured foliation $F\in \mathcal{MF}$, there is a unique holomorphic quadratic differential whose induced measured foliation is equivalent to $F$. 
\end{jthm}

\begin{figure}
  \centering
  \includegraphics[scale=0.35]{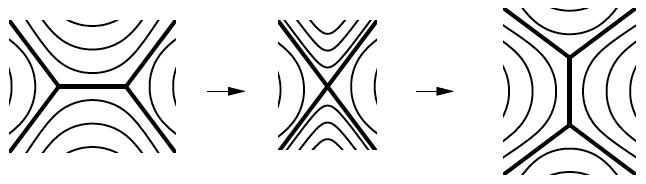}\\
 \caption{A Whitehead move preserves the equivalence class of a measured foliation.}
\end{figure}

\subsection{Structure at a double order pole}
In this article, we shall be concerned with quadratic differentials with poles of order two, that is,  when the local expression of the differential at the pole is:
\begin{equation}\label{ord2gen}
\left(\frac{a^2}{z^2} +  \frac{c_1}{z} + c_0 + \cdots \right)dz^2
\end{equation}
where $a \in \mathbb{C}^\ast, c_0, c_1 \in \mathbb{C}$. 
Note that  it is always possible to choose a coordinate $z$ where this expression is of the form (\ref{ord2}) - see \S7.2 of \cite{Streb}. Here  the complex constant $a$, with a choice of sign as defined in Definition \ref{defn:res}, is the \textit{residue} at the pole, and is easily seen to be independent of choice of coordinate chart. \\

\textbf{Induced measured foliation.} There is always a neighborhood $U$ of an order two pole such that any horizontal trajectory in $U$ continues to the pole in one direction, unless the trajectories close up to form closed leaves foliating the neighborhood. 
 We shall refer to such a neighborhood as a \textit{sink neighborhood} of the foliation.  See Chapter III, \S7.2 of \cite{Streb} for details. 

 There are three possibilities for what the induced foliation looks like in a sink neighborhood of an order two pole, depending on the residue, as defined in Definition \ref{defn:res}:
\begin{itemize}
\item $\Re(a) = 0$: The foliation is by concentric circles around the pole. In particular, each leaf in $U$ is closed.

\item $\Im(a) = 0$: The foliation is by radial lines from the pole. 

\item $\Re(a) \neq 0 \text{ and } \Im(a) \neq 0$: Each horizontal leaf  in $U$ spirals into the pole. 

\end{itemize}
See Figure 3 and section 4 of \cite{MulPenk} for examples.\\

The transverse measure of  a loop $\gamma$ linking the pole is computed by an integral as in (\ref{meas}). We obtain:
\begin{equation}\label{trans}
\tau(\gamma) =  \lvert \Im \left( \displaystyle\int_{\gamma} \sqrt{q}\right)\rvert  = \lvert \Im(2\pi \mathsf{i} a) \rvert  =  2\pi \Re (a) 
\end{equation}
and is thus determined by the real part of the complex residue. (Recall from Definition \ref{defn:res} that we have chosen the sign of the residue such that the real part is non-negative.) \\

\begin{figure}
  \centering
  \includegraphics[scale=0.35]{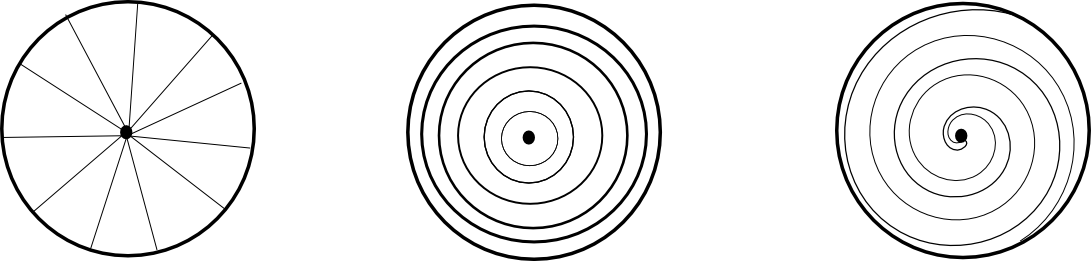}\\
 \caption{The induced foliation at a double order pole.}
\end{figure}

\textbf{Metric}.  By a similar computation using  (\ref{meas}), one can check that  in the induced singular-flat metric,  a sink neighborhood of the pole is a half-infinite Euclidean cylinder of circumference $2\pi \lvert a\rvert$, where $a \in \mathbb{C}^\ast$ is the residue.

\subsection{Harmonic maps to $\mathbb{R}$-trees}

Let $F$ be a measured foliation on a smooth surface $S$ (see Defn. \ref{mf}). The leaf space of its  lift $\tilde{F}$ to the universal cover, metrized by the projection along the leaves of its transverse measure, defines an $\mathbb{R}$-tree: namely, a simply connected metric graph  $T$ such that any two points of $T$ has  a unique geodesic (isometric embedding of a segment) between them. 
For definitions and details we refer to \cite{WolfT}; see also \cite{Kap} for more background on the construction and uses of $\mathbb{R}$-trees, and \cite{Sun}  for some useful analytical background, including restrictions on the shape of the image of a harmonic map.  See \cite{Schoen} for an insightful treatment of the following notion of Hopf differentials, and \cite{DaWen} for a very nice survey of both results and techniques in the uses of harmonic maps in \tec theory. Jost \cite{Jost2} provides detailed proofs of some of the results mentioned in the \cite{DaWen} survey, together with a tasteful choice of useful techniques and perspectives.

\begin{defn}[Harmonic map, Hopf differential]  Let $X$ be a (possibly non-compact)  Riemann surface. A map $h:X  \to T$ that has locally square-integrable derivatives  is \textit{harmonic} if it is a critical point of the energy functional
\end{defn}

\begin{equation}\label{endef}
\mathcal{E}(f)  = \displaystyle\int\limits_X\left( \lvert h_z\rvert^2 + \lvert h_{\bar{z}}\rvert^2 \right)  dzd\bar{z}
\end{equation}
in the case when this integral takes a finite value. 

When this energy is infinite, as in a number of situations in this paper, we adapt the definition to say a map is harmonic  when the above holds for the restriction of the energy functional to any relatively compact set $U$ in $X$. Namely,  the derivative of the above energy functional, {\it restricted to $U$}, vanishes for  any  variation of the map with compact support in $U$.\\

\textit{Notation.} In later sections, when we define equivariant maps on the universal cover of a surface, $\bar{\mathcal{E}}(\widetilde{f})$ shall denote the \textit{equivariant} energy, that is, the energy of the restriction of $\widetilde{f}$  to a fundamental domain of the action. (In particular, if the surface is compact, then the equivariant energy is finite.)\\

For a harmonic map $h$ with locally square-integrable derivatives, we define the \textit{Hopf differential} as:
\begin{equation*}
\text{Hopf}(h) = -4 (h_z)^2 dz^2.
\end{equation*}
This quadratic differential is then locally in $L^1$ on the Riemann surface, and Schoen \cite{Sch} points out the stationary nature of the map $h$ implies that this differential is weakly holomorphic, hence strongly holomorphic by Weyl's Lemma. Thus it is a classical holomorphic quadratic differential with isolated zeroes (if not identically zero) and smooth horizontal and vertical foliations. Now, in the convention chosen (see the remark below),
 the horizontal foliation relates to the harmonic map
 as integrating the directions of minimal stretch of the differential $dh$ of the harmonic map $h$, while the vertical foliation integrates the directions of maximal stretch of $dh$. (See e.g. \cite{Wolf0} for the easy computations which justify these geometric statements.)  Thus, when the harmonic map $h$ is 
 a projection to a tree, 
 the minimal stretch direction lies along the kernel of the differential map $dh$, and thus the horizontal leaves are the level sets of points in the tree. We see then that, away from the isolated zeroes of the Hopf differential, the harmonic map takes disks to geodesic segments in the $\mathbb{R}$-tree $T$.

\textit{Remark.} The constant in the definition above is chosen such that the Hopf-differential of the map $z\mapsto \Im z$ is $dz^2$, that is, the Hopf-differential realizes the horizontal foliation.  In particular, the sign convention  differs from the one used in, say \cite{Wolf2}, and in other places in the harmonic maps literature.\\

In particular, for a holomorphic quadratic differential $q$  on $\Sigma$, let $c= c(q):\tilde{\Sigma} \to T$ be the collapsing map to the leaf space of the lift of the measured foliation $\mathcal{F}(q)$, to the universal cover. Namely, the map takes each leaf to a unique point in $T$.  Recall that the $\mathbb{R}$-tree $T$ is metrized: distance is given by the pushforward of the transverse measure. The collapsing map is harmonic in the sense above: locally, it is the harmonic function $z\mapsto \Im(z)$,  and  the Hopf differential recovers the quadratic differential $q$, as a the reader can easily verify.  The energy (\ref{endef}) equals half the area of the domain in the singular-flat metric induced by the differential. As in the previous paragraph, geometrically the horizontal foliation is along the directions of least stretch for the harmonic map, and we shall refer to this as the \textit{collapsing foliation}. \\

The above observation has been used by one of us (\cite{Wolf2}) in a proof of the Hubbard-Masur Theorem (stated in \S2.1). Namely, for a Riemann surface $\Sigma$ and measured foliation $F$,  let $T$ be the leaf space of the lift of the foliation to the universal cover,  metrized by the transverse measure, and consider the Hopf differential of the energy-minimizing map $h:\tilde{\Sigma} \to T$  that is equivariant with respect to the surface group action on both spaces. To prove such an energy-minimizing harmonic map exists, one crucially needs the surface is compact, which provides a global energy bound. This Hopf differential of the energy-minimizer then descends to $\Sigma$ to give the required holomorphic quadratic differential.
In this paper, we adapt the strategy above for the corresponding measured foliations, and their harmonic collapsing maps of \textit{infinite energy};
more precisely, we study limits of sequences $h_n$ of harmonic maps from a compact exhaustion of a punctured surface. There will be no {\it a priori} bound on the total energies of the elements of these sequences, so we will need a finer analysis to control the shapes of the maps.

\section{Measured foliations with centers}

 Let $S$ be a smooth surface of genus $g$ with a set $P$ of marked points such that $\lvert P \rvert = n$.  We shall henceforth assume that  the Euler characteristic of the punctured surface is negative, namely $\chi(S\setminus p) = 2- 2g - n  <0$, so the surface admits a complete hyperbolic metric. 
 
 We refer to the previous section for the usual notion of measured foliations.

\begin{defn}[Center]\label{ctr}  A \textit{center} is a foliation on the  punctured unit disk $\mathbb{D}^\ast$ that is isotopic, relative to the boundary, to one of the two standard foliations:
\begin{enumerate}

\item{(Radial center)} A foliation by radial lines to the puncture, defined by the kernel of the smooth $1$-form $\Im(z^{-1}dz)$,
\item{(Closed center}) A foliation by concentric circles around the puncture, defined by the kernel of the smooth $1$-form $\Re(z^{-1}dz)$.
\end{enumerate}
\end{defn}

\textit{Remark.}  As discussed in \S2.2, a pole of order two as in (\ref{ord2}) induces a foliation on $\mathbb{D}^\ast$ that generically has leaves that spiral into the puncture. Any such spiralling foliation is isotopic in the punctured disk, to a radial center: namely, the smoothly varying family of foliations induced by
\begin{equation*}
\frac{(\Re a + (1-t) \Im a)^2}{z^2} dz^2\text{   for  } 0\leq t\leq 1 
\end{equation*}
has the initial foliation at $t=0$, and a radial center at $t=1$.

\begin{defn} A \textit{measured foliation on $S$ with centers at $P$} is  a smooth foliation on $S$ away from finitely many singularities such that
\begin{itemize}
\item a neighborhood of each point of $P$ is a center,
\item all other singularities are of prong type,  and
\item the foliation is equipped with a finite  positive measure on compact arcs on $S\setminus P$ transverse to the foliation, that is invariant under transverse homotopy.
\end{itemize}
\end{defn}

\textit{Remark.} Equivalently, a measured foliation as in the definition above can be described as follows. Consider an atlas of 
 charts $\{(U_i, \phi_i)\}$ away from $S \setminus (P \cup Z) $, where  $\phi_i$ are closed smooth $1$-forms  such that $\phi = \pm \phi_j $ on $U_i \cap U_j$. The foliation away from $P$ is locally defined to be the integral curves of  the line field associated with the kernel of the $1$-forms. Moreover, in a neighborhood of the the points of $Z$, the foliation is given by the kernel of the $1$-form $\Im(z^{k/2} dz)$ (see Figure 1 for the case $k=1$) and of the points of $P$, that the kernel of $\Im(z^{-1}dz)$ or $\Re(z^{-1}dz)$ (see Definition \ref{ctr}). The line element $\lvert \phi_i\rvert$ provides a well-defined transverse measure, that is invariant under transverse homotopy because the forms are closed.\\ 
In this description, the horizontal measured foliation induced by a meromorphic quadratic differential $q$ (\textit{cf.} Definition \ref{hf})  is obtained by taking the $1$-forms to be locally $\pm \Im(\sqrt q)$. \\

The main theorem of the paper can then be interpreted as a Hodge theorem for such smooth objects, namely, that any measured foliation with centers as above can be realized (up to equivalence) as one induced by a meromorphic quadratic differential.\\

In analogy with the space $\mathcal{MF}$ on a closed surface (\textit{cf.} Definition \ref{mf}) we have:

\begin{defn} Let $\mathcal{MF}_2(S,P)$ denote the space of measured foliations with centers, up to the equivalence relations of isotopy and Whitehead moves. 

(We shall drop the reference to the surface $S$ and set of centers $P$ when it is obvious from the context.)
\end{defn}

Namely, the space $\mathcal{MF}_2(S,P)$ is parametrized by an adaptation of the Dehn-Thurston coordinates for measured foliations that we mentioned in Definition \ref{mf}.We refer to Proposition 3.9 of \cite{ALPS} for details (see also \S3 of that paper), and provide a sketch of the argument.

\begin{prop}\label{spc} Let $S$ be a surface of genus $q \geq 2$ and  $P$ be a set of $n$ marked points, as above. The space  $\mathcal{MF}_2$ of measured foliations on $S$ with centers at $P$ is homeomorphic to $\mathbb{R}^{6g-6 + 3n}$.
\end{prop}

\begin{figure}
  \centering
  \includegraphics[scale=0.5]{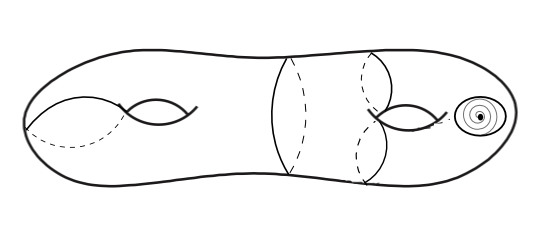}\\
 \caption{The space of measured foliations with a single center on a genus $2$ surface has dimension 9: each interior pants curve contributes two real parameters, and the closed curve around the center contributes one (Proposition \ref{spc})}
\end{figure}

\begin{proof}[Sketch of the proof] 

Given a measured foliation $F$ in $\mathcal{MF}_2$, we can excise the centers from $S$ and consider the measured foliation on the surface with boundary. It suffices to parameterize the space of such foliations, as $F$ is recovered (up to isotopy) by re-attaching a \textit{radial center} in the case the boundary has a positive transverse measure, and by a \textit{closed center} when the transverse measure vanishes. (See Definition \ref{ctr} for these notions.)  
In the former case, by an isotopy relative to the boundary, one can arrange that the leaves of $F$ on the surface-with-boundary are  orthogonal to the boundary component, and in the latter case, parallel to it. 

In either case, doubling the surface across the boundaries by an orientation-reversing reflection results in a doubled surface with a measured foliation.
Such measured foliations  on the doubled surface can then be parametrized by coordinates provided by the transverse measures and twisting numbers for a system of pants curves that  that have an  involutive symmetry, and include the curves obtained from doubling the boundary components  (\textit{cf.} \S3.4 of \cite{ALPS}).

The two-fold symmetry implies that the parameters for pairs of ``interior" pants curves (that is, other than ones obtained from the boundary curves) will agree. 

There are a total of $3g-3 + n$ interior pants curves (on the excised surface before doubling) and $n$ boundary curves.  Each boundary curve contributes one parameter, namely the transverse measure, to the foliation on the doubled surface.  Each interior pants curve contributes two (the length and twist parameters). Hence the total number of parameters determining $F$ is $6g-6 + 3n$. (See Figure 4 for an example.) Note that the boundary curve having transverse measure zero corresponds to the case of a closed center.

Conversely, by reversing the steps, any such collection of $6g-6 + 3n$ parameters determines a symmetric foliation on the doubled surface, and hence a unique foliation $F$ in $\mathcal{MF}_2$. Thus the space $\mathcal{MF}_2$ is homeomorphic to $\mathbb{R}^{6g-6+3n}$. \end{proof}

\textit{Remark.} Measured foliations with $1$-prong singularities  (associated with poles of order one) arise in the context of classical Teichm\"{u}ller theory for a surface with punctures. For these, the usual parametrization  for closed surfaces extends by associating two real parameters for each such singularity. Namely, the corresponding space $\mathcal{MF}_1$ of foliations is of dimension $6g-6+2n$ (see Expos\'{e} XI of \cite{FLP}). For $\mathcal{MF}_2$ as in the proposition above, the additional real parameter is the transverse measure of a loop around each center.

\section{Model maps and cylinder ends}

As discussed in \S2, a meromorphic quadratic differential induces a measured foliation and  a harmonic collapsing map to its metrized leaf-space. Recall that a  quadratic differential with a double order pole has a neighborhood that is a ``center" (see Definition \ref{ctr}). The leaf space of the foliation on a radial center is a circle of circumference given by the transverse measure around the pole, and in the case of a closed center,  a copy of $\mathbb{R}_{\geq 0}$. (See Figure 5.) 
In this section we describe a model family of harmonic maps from a neighborhood of a pole  to $S^1$ (or $\mathbb{R}_{\geq 0}$), parametrized by the complex residue at the pole.
 Up to bounded distance, these form all the possibilities for the restriction of the harmonic collapsing map to that neighborhood.  \\

\textit{Notation.} Throughout this paper,  $S^1_R$ shall denote a circle of circumference $R$, and $\mathcal{C}_R$ the half-infinite Euclidean cylinder of circumference $R$:
\begin{equation*}
\mathcal{C}_{R} = \{ (x, \theta) \vert x\in [0,\infty), \theta \in S^1_{R}\}.
\end{equation*}
A finite subcylinder $(0\leq x\leq L)$ of $\mathcal{C}_{R}$ is denoted by $\mathcal{C}_R(L)$.\\

\subsection*{A basic example}
The differential  $\frac{1}{z^2}dz^2$ on the punctured disk  $\mathbb{D}^\ast$ yields a half-infinite Euclidean cylinder  $\mathcal{C}_{2\pi}$ in the induced singular-flat metric.  In fact, in the coordinates $(x,\theta)$ on this flat cylinder, the differential is $d\omega^2$ for the complex coordinate $\omega = x+i\theta$.
The induced foliation then comprises the  ``horizontal rays"  ($\theta = \textit{const.}$) along the length of the cylinder, and the collapsing map is  
\begin{equation*}
h:\mathcal{C}_{2\pi} \to S^1_{2\pi} \text{   where  } h(x,\theta) = \theta.
\end{equation*}
that is clearly harmonic. \\

Other model maps are then obtained by introducing ``twists" to this basic example above,  together with scaling the circumference.  This amounts to multiplying the quadratic differential $dw^2$ on $\mathcal{C}_{2\pi}$ by a non-zero complex number (\textit{cf.} Lemma \ref{resmod}). The entire family of model maps also includes the  ``limiting" case of a closed center, which is obtained as the ``twists" tend to infinity, or equivalently, the angle of the leaves with the longitudinal direction tends to $\pm \pi/2$. 

\subsection{Definitions}\label{mmap} For $R\in \mathbb{R}_+$ and $\alpha \in (-\pi/2, \pi/2)$, let $\mathcal{C}_R$ be a half-infinite Euclidean cylinder of radius $R$ as above, and define $m_{R,\alpha}:\mathcal{C}_R\to S^1_{R\cos\alpha}$ to be the map
\begin{equation}\label{lift}
m_{R,\alpha}(x, \theta) = \theta\cos\alpha + x\sin\alpha 
\end{equation}
where the right-hand side is considered modulo $R\cos\alpha$. This map is harmonic: in local Euclidean coordinates on the cylinder, it is a linear function and hence its Laplacian vanishes.

Here the angle $\alpha$ shall be referred to as the \text{foliation angle}: this denotes the angle from the  longitudinal direction,  of the leaves of the collapsing foliation for the model map. 

For $\alpha = \pm \pi/2$, we define $m_{R,\alpha}:\mathcal{C}_R\to \mathbb{R}_{\geq 0}$ by $(x,\theta) \mapsto x$.  This is the case corresponding to when the collapsing foliation ``degenerates" to a foliation of $\mathcal{C}_R$ by closed circles.

\begin{figure}
  \centering
  \includegraphics[scale=0.7]{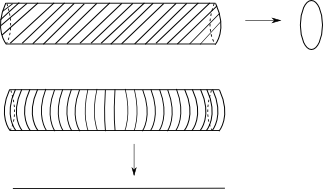}\\
 \caption{The collapsing map for the induced foliation in a neighborhood of an order two pole yields a harmonic map to $S^1$ or $\mathbb{R}$.}
\end{figure}

Together, we obtain 

\begin{defn}[Model maps] The expression (\ref{lift}) defines the family 
\begin{equation*}
\mathsf{M} = \{ m_{R,\alpha} \vert R \in \mathbb{R}_+, \alpha \in S^1_{\pi}\}
\end{equation*}
of harmonic maps from $\mathcal{C}_R$ to a target that is either $S^1$ (of circumference $R\cos\alpha$, when $\cos\alpha \neq 0$) or $\mathbb{R}$ (when $\cos\alpha = 0$).

Note that the interval $[-\pi/2, \pi/2]$ that is the range of values for $\alpha$  has been identified with a circle $S^1_{\pi}$, where $\pm \pi/2$ defines a single point.

As in the basic example at the beginning of the section, these model maps $m_{R, \alpha}$ are the collapsing maps for a quadratic differential on $\mathbb{D}^\ast $ of the standard form (\ref{ord2}), and are parameterized by its complex residue $a$. (For the relation between the parameters $R, \alpha$ with $a$ see Lemma \ref{resmod}.) 

\end{defn}

\begin{defn}[Model end] \label{modelend} For a choice of parameters $R \in \mathbb{R}_+$ and $\alpha \in S^1_{\pi}$, we shall refer to the half-infinite cylinder $\mathcal{C}_R$ with the foliation at angle $\alpha$ described above, as the \textit{model cylindrical end} determined by those parameters.

Since the metric on the leaf-space of a measured foliation is determined by transverse measures, the transverse measure of an arc on a model cylindrical end  is equal to the length of the image segment under the collapsing map (\ref{lift}) to the leaf-space.  In particular, the transverse measure of a meridional circle is $R\cos\alpha$, and the transverse measure of a longitudinal segment of length $L$  (such as across a truncated cylinder $\mathcal{C}_R(L)$, is $L\sin \alpha$.  
(The latter quantity is considered modulo $R\cos\alpha$, which is the circumference of the image circle, but is exactly $L\sin \alpha$ when we consider the lift to the universal cover in the next subsection.) 

\end{defn}

\subsection{Lifts to the universal cover}

We fix the parameters $R, \alpha$ for the rest of this section. 

Let $\mathbb{H} = \{x,\theta \vert x\in \mathbb{R}_+, \theta \in \mathbb{R}\}$ be the universal cover of the interior of $\mathcal{C}_R$. 

A model map $m \in \mathsf{M}$ lifts to a map 
\begin{equation*}
\widetilde{m}: \mathbb{H} \to \mathbb{R}
\end{equation*}
which is equivariant with respect to the $\mathbb{Z}$-action that acts 
\begin{itemize}
\item[(a)] on the target by translation by $R\cos\alpha$ when $\alpha \neq \pi/2$, and trivially when $\alpha  = \pm \pi/2$,  and  
\item[(b)] on the domain by a translation in the $\theta$-coordinate by $R$.
\end{itemize}

Any quadratic differential on $\mathbb{D}^\ast$ with a pole of order two at the puncture also has an induced foliation that lifts to the universal cover $\mathbb{H}$. (The universal covering map $\pi:\mathbb{H} \to \mathbb{D}^\ast$ is the quotient by the translation $z\mapsto z+R$.) 
The following lemma asserts that the corresponding collapsing map is bounded distance from \textit{some} model map in $\mathsf{M}$. \label{rerranged Lemmas, added some sentences}

\begin{lem}\label{resmod}  Let $q$ be a meromorphic quadratic differential on $\mathbb{D}^\ast$ with an order two pole of residue $a\in \mathbb{C}^\ast$. Then the collapsing map of the induced foliation on $\mathbb{H}$  is bounded distance from the lift of the model map $m_{R,\alpha}$ where $R = 2\pi \lvert a\rvert$ and $\alpha = \text{Arg}(a)$.
\end{lem}
\begin{proof}
For a pole of the form (\ref{ord2})  the collapsing map is identical to $m_{R,\alpha}$, using the fact that the ``spiralling" ray
\begin{center}
$r(t) = (t\cos \alpha,  t\sin \alpha)$ 
\end{center}
is a horizontal trajectory on the cylinder $\mathcal{C}_R$ with the usual $(x,\theta)$-coordinates, as one can check by  a straightforward calculation. 

 In general, the remaining terms in the Laurent expansion (\ref{ord2gen}) contribute to a bounded perturbation of the collapsing map, since by (\ref{meas})  the additional transverse measure is bounded by 
\begin{equation*}
\left\vert \displaystyle\int\limits_0^r \left(\frac{c_1}{z} + c_0 + f(z)\right)^{1/2} dz \right\vert   =  O\left(\displaystyle\int\limits_0^r z^{-1/2} dz\right) = O(1)
\end{equation*}
for any $0<r<1$, where $f$ is a  holomorphic function on $\mathbb{D}$ that vanishes at the origin.

This transverse measure determines distances in the leaf-space that is the image of the collapsing map. Hence the above observation implies that up to a bounded error, the image is determined by the leading order  $(a^2/z^2)$ term in (\ref{ord2gen}), which as we previously noted, is identical to $m_{R,\alpha}$.

(Alternatively, note that with respect to the coordinate $\omega=x+i\theta$ on the half-infinite flat cylinder, the  expression (\ref{ord2gen}) is $(a^2 + c_1e^{-\omega})d\omega^2$: the exponentially decaying term contributes a bounded amount to the associated collapsing map.)
\end{proof}

The following Lemma now  asserts that the model map in Lemma \ref{resmod} at a bounded distance from the collapsing map, is uniquely determined, if one also knows the parameter $R$. 

\begin{lem}\label{dist} Two model maps $m_1,m_2 \in \mathsf{M}$ lift to maps
\begin{equation*}
\widetilde{m_1},\widetilde{m_2}:\mathbb{H} \to \mathbb{R}
\end{equation*}
that are a bounded distance apart if and only if their foliation angles $\alpha_1$ and $\alpha_2$ are equal.
\end{lem}

\begin{proof}
The `if' direction is straightforward.  The model maps $\widetilde{m_i}$ are determined by the two parameters of the radius $R_i$ and the angle $\alpha_i$. If the foliation angles $\alpha_1$ and $\alpha_2$ and the radii $R_1$ and $R_2$ are equal, then by the expression (\ref{lift}) the two lifts are in fact identical, up to a translation of base point. Further, if the angles $\alpha_1$ and $\alpha_2$ are identical, but the maps differ in their parameter $R$, then the lifts are still the same (and it is only the translation action of $\mathbb{Z}$ on $\mathbb{H}$ and $\mathbb{R}$ that differs).

For the other direction, consider the case when neither foliation angle equals $\pm \pi/2$ (which is a single point in the parameter space for angles, $S^1_\pi$, that corresponds to a closed center). Then the expression (\ref{lift}) also represents the lift of the model map to $\mathbb{R}$, where the latter is thought of as the universal cover of $S^1_{R\cos\alpha}$.  From (\ref{lift}), we see that two such maps  with different values of $\alpha$, say $\alpha_1$ and $\alpha_2$, are an unbounded distance apart: consider, for example, the restriction to a $\theta= \textit{const.}$ vertical line. Then the distance between the images  equals $x\lvert \sin \alpha_1-\sin \alpha_2\rvert \to \infty$ as $x\to \infty$. 

If one of the foliation angles is $\pm \pi/2$, then the image of an entire horizontal line is bounded (it maps to a point in the target $\mathbb{R}$) which, again by (\ref{lift}),  is not the case for $\alpha \neq \pm \pi/2$. 
\end{proof}

We shall use the previous two lemmas in the \S5.2.4, the endgame of the proof, where we need to verify that the harmonic map we produce, at a bounded distance from a given model map near the puncture, indeed has the required residue.

\subsection{Least-energy property}

We begin by noting the energy of a model map, that we will need in the next section.

\begin{lem}\label{ener} The restriction of the model map $m_{R,\alpha}:\mathcal{C}_R \to S^1_{R\cos\alpha}$ to a sub-cylinder $\mathcal{C} = \mathcal{C}_R(L)$ has energy equal to $\frac{1}{2}RL = \frac{1}{2} R^2 \cdot \text{Mod}(\mathcal{C})$. 
\end{lem}
\begin{proof}
The expression for energy (\ref{endef}) can be converted to Euclidean $(x,\theta)$-coordinates on $\mathcal{C}$ to obtain:
\begin{equation*}
\mathcal{E}(f)  = \frac{1}{2} \displaystyle\int\limits_\mathcal{C}\left( \lvert h_x\rvert^2 + \lvert h_{\theta}\rvert^2 \right)  dxd{\theta}
\end{equation*}
and then using the expression (\ref{lift}) of the map, we see that the integrand is $1$. Hence the energy equals half the Euclidean area of the cylinder, that is, $\frac{1}{2}RL$.

Note that the modulus of the cylinder $\text{Mod}(\mathcal{C}) = L/R$. \end{proof}

What we shall now prove is a crucial property of a model map, namely its restriction to subcylinders solves a certain least-energy problem for maps to an $\mathbb{R}$-tree target. \\

In what follows, let $A$ be an annulus of modulus $0<M<\infty$ that we uniformize to the flat cylinder $\mathcal{C}_R(L)$, that is, of  circumference $R$ and length $L$, such that $L/R=M$. 
As usual, we equip this cylinder with coordinates  $x,\theta$.

Let  $\widetilde{A} = [0,L] \times \mathbb{R}$ be the universal cover of $A$, and $T$ be an $\mathbb{R}$-tree with a $\mathbb{Z}$-action (that is, in fact, part of a larger surface-group action, which we shall not need). 

Consider a map
\begin{equation*}
\tilde{u}: \widetilde{A} \to T
\end{equation*}
which is $\mathbb{Z}$-equivariant, where the action on the domain is by translation by $R$ on the $\mathbb{R}$-factor. 
 Let $D = [0,L] \times [0,R] \subset \widetilde{A}$ be a fundamental domain of the action.

\begin{prop}\label{minim} Let 
\begin{center}
${u}:[0,L]\times[0,R] \to T$
\end{center}
be a map to an $\mathbb{R}$-tree, defined on a fundamental domain in the universal cover of an annulus of modulus $M := L/R$ as defined above. Suppose that there exist constants $C,\tau>0$ such that ${u}$ satisfies:
\begin{itemize}
\item[(1)]  each vertical ($\theta$-) arc in $D$ maps to a segment of length $C$ in $T$, and 
\item[(2)]  each horizontal ($x$-) arc in $D$ maps to a segment of length $\tau$ in $T$.
\end{itemize}
Then the  energy of the map  satisfies the lower bound
\begin{equation}\label{lowerb}
\mathcal{E}(u) \geq \frac{1}{2}M \cdot \left(C^2+ \frac{\tau^2}{M^2}\right).
\end{equation}

\end{prop}

\begin{proof}
We adapt the classical length-area argument.

Since every vertical arc  in $D= [0,L] \times [0,R]$ maps to an arc of length $C$, we have
\begin{equation*}
\displaystyle\int\limits_0^{R}  u_\theta d\theta \geq C
\end{equation*}
Integrating over the longitudinal ($x$-) direction we obtain:
\begin{equation*}
\displaystyle\int\limits_0^{R}\int\limits_0^{L}  \lvert u_\theta\rvert dxd\theta  \geq LC.
\end{equation*}
By the Cauchy-Schwarz inequality,
\begin{equation*}
 \displaystyle\int\limits_0^{R} \int\limits_0^{L}  \lvert u_\theta \rvert^2  dx d\theta  \displaystyle\int\limits_0^{R}\int\limits_0^{L} 1^2  dxd\theta \geq L^2 C^2
\end{equation*}
which implies 
\begin{equation}\label{ut}
 \displaystyle\int\limits_0^{R} \int\limits_0^{L}  \lvert u_\theta \rvert^2  dx d\theta   \geq  \frac{L^2 C^2}{LR} = C^2M.
\end{equation}

On the other hand, for each horizontal  arc we have:
\begin{equation*}
\displaystyle\int\limits_0^{L}  u_x dx \geq \tau
\end{equation*}
and by the same calculation as above we obtain
\begin{equation}\label{ux}
 \displaystyle\int\limits_0^{R} \int\limits_0^{L}  \lvert u_x \rvert^2  dx d\theta   \geq \frac{\tau^2}{M}.
\end{equation}

Combining (\ref{ut}) and (\ref{ux}) we find:

\begin{equation*}
\mathcal{E}(u)  = \frac{1}{2}  \displaystyle\int\limits_0^{R} \int\limits_0^{L}  \left(\lvert u_x \rvert^2  +    \lvert u_\theta \rvert^2  \right)dx d\theta  \geq   \frac{1}{2}M \cdot \left(C^2+ \frac{\tau^2}{M^2}\right)
\end{equation*}
as desired.

\end{proof}

  \textit{Remark.} The restriction of a model map $m =m_{R,\alpha}:\mathcal{C}_R \to S^1_{R\cos\alpha}$ to a subcylinder $\mathcal{C} = \mathcal{C}_R(L)$ lifts to a $\mathbb{Z}$-equivariant map
  \begin{equation*}
  \tilde{m}: [0,L] \times \mathbb{R} \to \mathbb{R}
  \end{equation*}
  that satisfies (1) and (2) with $\tau = L\sin \alpha$ and $C=R\cos\alpha$, which are the transverse measures of the foliation across the subcylinder (see Definition \ref{modelend}). 
   
   A simple substitution in (\ref{lowerb}) then yields $\mathcal{E}(\tilde{m}\vert_{\tilde{C}}) \geq \frac{1}{2}R^2\cdot  \text{Mod}(\mathcal{C})$.  Lemma \ref{ener} then confirms that, in fact, equality holds, and so the lift of the model map realizes the least energy for all  $\mathbb{Z}$-equivariant maps from $\widetilde{\mathcal{C}}$ to $\mathbb{R}$ satisfying (1) and (2).

\section{Proof of Theorem \ref{main}} 

We fix throughout the Riemann surface $\Sigma$ with a single marked point that we denote by  $P$.
Note that in the case of several marked points, we need add only an additional subscript, and the entire argument holds \textit{mutatis mutandis}.

We also fix a measured foliation $F \in \mathcal{MF}_2(\Sigma, P)$ with center  $U$ around $P$. Let $T$ be an $\mathbb{R}$-tree that is the leaf-space of the lifted foliation $\widetilde{F}$ on the universal cover. Recall that this acquires a metric from the projection of the transverse measures of the foliation.
The collapsing map to the leaf space is then a map from $\widetilde{\Sigma \setminus P}$ to $T$.

The proof of Theorem \ref{main} involves showing there is a \textit{harmonic} map 
\begin{center}
$\widetilde{h}:\widetilde{\Sigma \setminus P}\to T$ 
\end{center}
whose collapsing foliation is $F$, with a Hopf differential with a prescribed double order pole at $P$. We begin with the (easier) proof of the uniqueness of such a map in \S5.1, and then show existence in \S5.2.

\subsection{Uniqueness} 

Let $q_1$ and $q_2$ be meromorphic quadratic differentials realizing the same measured foliation $F$, and with the same imaginary part of the residue at $P$ (up to sign).

These determine harmonic maps $\widetilde{h_1},\widetilde{h_2}: \widetilde{\Sigma \setminus P}\to T$. 

Recall from \S1 that the real part of the residue at the pole is determined by $F$, namely, it is the transverse measure of a loop around the pole. Our convention is that this real part is always taken to be non-negative.

Since we have assumed that $q_1$ and $q_2$ have the same imaginary parts, the (complex) residues for $q_1$ and $q_2$ at the pole $P$ are equal.

By Lemma \ref{resmod}, this residue then determines the (lift of the) model map $m_{R,\alpha}$ to which $\widetilde{h_1}$ and $\widetilde{h_2}$ are asymptotic.  (In particular, $R$ is the circumference of the half-infinite cylinder in the induced metric, that is, equals $2\pi \lvert \text{Res}(P)\rvert$.)  That is,  both $\widetilde{h_1}$ and $\widetilde{h_2}$  are bounded distance from the lift $\widetilde{m_{R,\alpha}}$. Hence the distance function between $\widetilde{h_1}$ and $\widetilde{h_2}$ is bounded. The distance function is invariant under the action of a deck transformation, hence descends to the (original) Riemann surface.  However it is well-known (this is an application of the chain-rule - see \cite{Jost2} for Riemannian targets and \cite{KorSch1} or \cite{Wolf2} for tree-targets)  that  such a distance function is subharmonic, and since a punctured Riemann surface is parabolic in the potential-theoretic sense, the distance must be constant. Moreover, the constant  is zero because of a local analysis around a prong singularity, as in the proof of Proposition 3.1 in \cite{DaDoWen} (see also section 4 of \cite{Wolf2}).

\subsection{Existence}
In this section, we show that given 

\begin{itemize}
\item  the topological data of a measured  foliation $F\in \mathcal{MF}_2$ with a center at $P$,  and
\item a complex residue $a$ at the pole $P$ that is compatible with the transverse measure of the linking loop around it, namely, having the corresponding real part, 
\end{itemize}
there exists a quadratic differential on the punctured surface $\Sigma \setminus P$ with a double order pole at $P$ of that residue (or its negative), realizing $F$. 

Recall that $T$ is the leaf-space of the lift of the foliation $\mathcal{F}$ to the universal cover.
The required quadratic differential shall be obtained as a Hopf differential of a harmonic map $\widetilde{h}: \widetilde{\Sigma \setminus P} \to T$, which in turn is obtained as a limiting of a sequence of harmonic maps $\widetilde{h_n}$ defined on (lifts) of a compact exhaustion of the punctured surface.\\

The following is an outline of the proof:

\begin{itemize}

\item \textit{Step 1.} We  start with $\Sigma \setminus P$ equipped with $F$, such that  on the neighborhood $U$ of the puncture, the foliation $F$ restricts to that of a model cylindrical end for the parameters $R$, $\alpha$. Consider a compact exhaustion $\Sigma_n$ of $\Sigma \setminus P$. For each $n$, by a holomorphic doubling across the boundary $\partial \Sigma_n$, we use the Hubbard-Masur Theorem to produce a holomorphic quadratic differential on $\Sigma_n$  that realizes the restricted foliation $F\vert_{\Sigma_n}$. Lifting to the universal covers, we obtain harmonic maps $\widetilde{h_n}:\widetilde{\Sigma_n} \to T$. What is crucial is that they are the energy-minimizing maps for their Dirichlet boundary conditions. \\

\item  \textit{Step 2.} Consider the annular region $A_n := \Sigma_n \setminus \Sigma_0$. We show that the energies of the restrictions of $\widetilde{h_n}$ to the lifts of $A_n$ have a lower bound proportional to the modulus of $A_n$. Together with their energy-minimizing property, an argument adapted from \cite{Wolf3} then provides  a uniform energy bound on the compact subsurface $\Sigma_0$.  Standard arguments then apply, and show the uniform convergence of $\widetilde{h_n}$ to a harmonic map $\widetilde{h}:\widetilde{\Sigma \setminus P} \to T$. \\

\item  \textit{Step 3.} We verify that the Hopf differential of $\tilde{h}$  yields a holomorphic quadratic differential on the quotient surface $\Sigma \setminus P$   with a double order pole at $P$ having the required residue, and realizing the desired foliation $F$.\\

\end{itemize}

We shall carry out the above steps in the following sections. 

\subsubsection{Step 1: The approximating sequence $\widetilde{h_n}$} 

Let $a\in \mathbb{C}^\ast$ be the desired residue at $P$.  Associated with this complex number are the parameters of circumference $R$ and ``foliation angle" $\alpha$  (see Lemma \ref{resmod}) and the model cylindrical end for $F$ around $P$. In particular, the parameter $R$ equals $2\pi$ times the absolute value of the residue.

Choose a coordinate disk $U$ around $P$. 
After an isotopy if necessary, we can assume the foliation $F$ restricts to the foliation on a model cylindrical end on $U$, for the parameters $R$ and $\alpha$, as in Definition \ref{modelend}. \\

\begin{figure}
  \centering
  \includegraphics[scale=0.5]{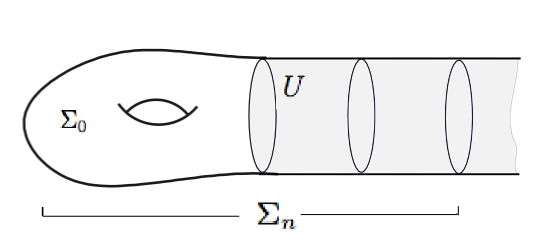}\\
 \caption{Exhaustion of the punctured surface $\Sigma \setminus P$.}
\end{figure}

Let $\Sigma_0$ be the Riemann surface with boundary  obtained by excising $U$.  

For $n\geq 1$ consider a compact exhaustion $\Sigma_n$  of $\Sigma$ obtained by excising a succession of increasingly smaller disks in $U$ around $P$. (See Figure 6.) 

As noted in the outline, we set $A_n = \Sigma_n \setminus \Sigma_0$ to be the intermediate annulus. 

Let ${F_n}$ be the restriction of the measured foliation $F$ to the compact surface-with-boundary $\Sigma_n$. By construction, on the annulus $A_n$, this foliation is the restriction of that on a model cylindrical end.

\subsection*{Doubling.} 
Let $n\geq 1$.
Take two identical copies of $\Sigma_n$, and perform a holomorphic doubling across $\partial \Sigma_n$: namely, take  two identical copies of the resulting surface-with-boundary $S_0$, and glue corresponding boundary components by an orientation-reversing homeomorphism of the circle, with exactly two fixed points (see Figure 7).
The resulting doubled surface $S_d$  has an orientation-preserving involution $\mathcal{I}:S_d\to S_d$.

Using this involution, one can push forward a  measured foliation on $S_0$ to one on $S_d$ that is symmetric under $\mathcal{I}$. 

[Note that this construction differs from the usual ``doubling-by-reflection" in the proof of Proposition \ref{spc} that fixes each boundary component and has an orientation-reversing involution. The advantage of the present construction is to ensure that the foliation on $S_0$,  that can be isotoped to be at a constant angle $\alpha$ at the boundary, extends smoothly to the doubled surface $S_d$ maintaining the twists about the boundary curves.]

We obtain a compact Riemann surface 
\begin{equation} \label{eqn:hatdecomp}
\widehat{\Sigma_n} = \Sigma_n \sqcup_\partial {\Sigma}_n
\end{equation}
with a holomorphic involution $\mathcal{I}:\widehat{\Sigma_n}  \to \widehat{\Sigma_n} $ that takes one copy of $\Sigma_n$ to the other. 

We define  the corresponding ``doubled" measured foliation $\widehat{F_n}$ on $\widehat{\Sigma_n}$. Note that by the construction of $F$ on $U$, the foliation $F_n$ is incident at a constant angle $\alpha$  at the boundary, and hence the foliation  $\widehat{F_n}$  is smooth across the doubled boundary. (See Figure 9.)

Moreover, the foliation $\widehat{F}_n$ is defined by the same parameters of twist and intersection-number on the pants curves in the interior of $\Sigma_n$, and its restriction  to the central cylinder $\hat{A}_n$ obtained by doubling $A_n$ is equivalent to the foliation on a model cylindrical end for parameters $R, \alpha$.  
 
 \begin{figure}
  \centering
  \includegraphics[scale=0.45]{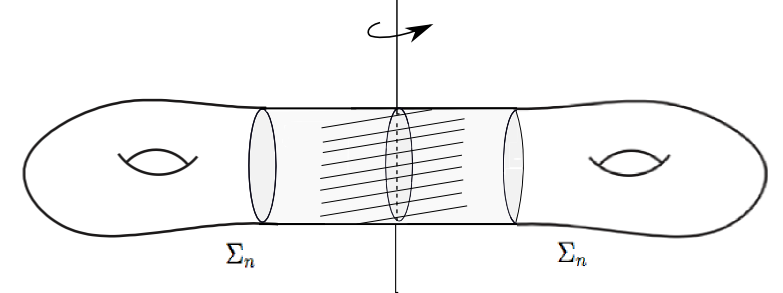}\\
 \caption{The foliation $\hat{F}_n$ is invariant under the holomorphic involution $\mathcal{I}$, indicated here as a $\pi$-rotation about a central axis.}
\end{figure}
 That is, for $n\geq 1$, the restriction of $\widehat{F_n}$ to this annulus is the foliation of the  finite subcylinder $\mathcal{C}_R(2L_n)$ of a model cylindrical end for the parameters $R$ and $\alpha$, where 
\begin{equation}\label{ln}
\text{mod}(A_n) = L_n/R.
\end{equation}
is the modulus of the subcylinder. 
 
 The transverse measure of a meridional curve around the cylinder is $L_n\cos \alpha$, and that of a longitudinal arc across it is $L_n \sin\alpha$ (see Definition \ref{modelend}). Note that $L_n\to \infty$ as $n\to \infty$.

 The involution $\mathcal{I}$ preserves the foliation $\widehat{F_n}$,  since 
 the foliation $\widehat{F_n}$ was defined on $\hat{\Sigma_n}\setminus\Sigma_n$ to agree with $\mathcal{I}(\widehat{F_n}|_{\Sigma_n})$, and $\mathcal{I}$ squares to be the identity.

Now, by the Hubbard-Masur Theorem (see \S2.1), there exists a holomorphic quadratic differential $\hat{q_n}$ that realizes the measured foliation $\widehat{F_n}$ on the doubled surface $\widehat{\Sigma}_n$.

Since the involution $\mathcal{I}$ preserves the foliation $\widehat{F_n}$ of $\hat{q_n}$, by the uniqueness part of the Hubbard-Masur theorem, the holomorphic involution $\mathcal{I}$ takes $\hat{q_n}$ to itself, namely $\mathcal{I}_\ast \hat{q_n} = \hat{q_n}$. 

Therefore taking the quotient by the holomorphic involution $\mathcal{I}$ on $\widehat{\Sigma}_n \setminus \partial\Sigma_n$, we obtain a (well-defined) holomorphic quadratic differential $q_n$ on $\Sigma_n$ (that in fact equals the restriction of $\hat{q_n}$ on a copy of $\Sigma_n \subset \widehat{\Sigma}_n$). Its induced foliation $\mathcal{F}(q_n)$  realizes the (well-defined quotient) measured foliation $F_n$ on the surface.

We now identify a foliated annular subsurface which converges to the spiralling end in the limiting construction in the final sections of the paper:

\begin{lem}\label{truncc} There is an embedded annulus $A(q_n)$ on $\Sigma_n$ such that the foliation $\mathcal{F}(q_n)$ restricted to $A(q_n)$ is identical, via a  leaf-preserving diffeomorphism, to a truncation of a model cylindrical end with parameters $R$ and $\alpha$, and that has modulus $\text{mod}(A_n) = L_n/R$.
\end{lem} 
\begin{proof}
On the doubled surface $\hat{\Sigma}_n$, the (doubled) measured foliation $\hat{F}_n$ has a  foliated cylinder  $\mathcal{C}_R(2L_n)$ that is a truncation of a model cylindrical end, coming from doubling the restriction of $F_n$ to $A_n$.  Now, the foliation $\mathcal{F}(\hat{q}_n)$ is measure-equivalent to $\hat{F}_n$, and differs from it by at worst isotopy and Whitehead moves.
Thus, the foliation $\mathcal{F}(\hat{q}_n)$ still contains an annulus, denoted by ${A}(\hat{q}_n)$, which is measure equivalent to $\mathcal{C}_R(2L_n)$; that is, there is a leaf-preserving diffeomorphism of the foliated annulus ${A}(\hat{q}_n)$ to $\mathcal{C}_R(2L_n)$, such that the transverse measure around the annulus is $R \cos\alpha$, and any arc between the boundary components has transverse measure $2L_n \sin \alpha$; measurements that agree with those in a truncation of a model end (see Definition \ref{modelend}).
By the two-fold symmetry of the foliation under the involution $\mathcal{I}$, the foliated annulus ${A}(\hat{q}_n)$ descends to an annulus $A(q_n)$ on $\Sigma_n$ that is measure equivalent to the truncation $\mathcal{C}_R(L_n)$ of a model cylindrical end. 
\end{proof}

\begin{definition}[Truncated center]
The foliated cylinder $A(q_n)$ of Lemma~\ref{truncc} above will be referred to as the ``truncated center" of the quadratic differential $q_n$ in the Riemann surface $\Sigma_n$. 
[When the context is clear, we shall often just refer to $A(q_n)$ as the "truncated center".]	
\end{definition}

The transverse measures, across the  truncated center $A(q_n)$ or around it, are identical to those across or around $\mathcal{C}_R(L_n)$ of the model cylindrical end with parameters $R$ and $\alpha$. A crucial component of the proof of convergence is to control the position of the truncated center relative to the initial foliated annulus $A_n$ (see \S5.2.2).

\subsection*{The harmonic map $\widetilde{h_n}$.} 

 By considering the lift $\widetilde{q_n}$ of the holomorphic quadratic differential ${q_n}$ to the universal cover $\widetilde{\Sigma}_n$ and the collapsing map of its induced measured foliation, one obtains a harmonic map  $\widetilde{{h}_n}$ which maps to the leaf-space of the lift of $F_n$. That image leaf-space is a sub-$\mathbb{R}$-tree of the leaf space $T$ of $\widetilde{F}$, namely, the leaf-space of the restriction of the foliation to  $\widetilde{\Sigma}_n$ that we henceforth denote by $T_n$. Thus we have:
 \begin{equation*}
 \widetilde{h_n}:  \widetilde{\Sigma}_n \to T_n \subset T.
 \end{equation*}

By the harmonicity and the uniqueness of such (non-degenerate) harmonic maps to $\mathbb{R}$-trees (or more generally, non-positively curved spaces, see \cite{Mese}),  this map solves the least energy (Dirichlet) problem for an equivariant map to the tree $T_n$ with the given boundary conditions.

\subsubsection{Step 2: Energy restricted to the annulus $A_n$}

The purpose of this subsection is to establish a lower bound on the energy of $\widetilde{h}_n$ when restricted to lifts of the annulus $A_n$.

Recall from Lemma \ref{truncc} that on the doubled surface, a cylinder with a spiralling foliation persists and descends to a foliated  embedded annulus $A(q_n)$ on $\Sigma_n$, with one boundary component $\partial \Sigma_n$. This is what we call the ``truncated center".

Note that since the truncated center $A(q_n)$ is a subcylinder of a model cylindrical end with parameters $R$ and $\alpha$, the transverse measure of a longitude is $L_n \sin \alpha$, where $L_n$ is as in (\ref{ln}), and that of a meridional circle is $R\cos\alpha$ (see Definition \ref{modelend}). \\ 

Let $\gamma$ denote the boundary curve  $\partial \Sigma_n$ that defines an essential  closed curve on the doubled surface; thus $\gamma$ is also the core of the cylinder $\hat{A}_n$.  (Topologically, this is the same curve for all $n$, and hence has no subscript ``$n$".) 

Similarly, let $\hat{\beta}$ denote a non-trivial  curve on  $\widehat{\Sigma_n}$,  symmetric under the involution $\mathcal{I}$,  that intersects $\gamma$ twice (see Figure 8) and let $\beta$ be the embedded arc in $\Sigma_n$ obtained under the quotient by the involution.
 In what follows, we shall use the transverse measures of these curves to control the position of the annulus $A(q_n)$  relative to $A_n$. \\

\textit{Notation.} In what follows $\tau_n(\sigma)$ denote the transverse measure of a topologically non-trivial embedded closed curve or arc $\sigma$ on $\Sigma_n$, where the transverse measure is with respect to the measured foliation  ${F_n}$.

\begin{figure}
  \centering
  \includegraphics[scale=0.5]{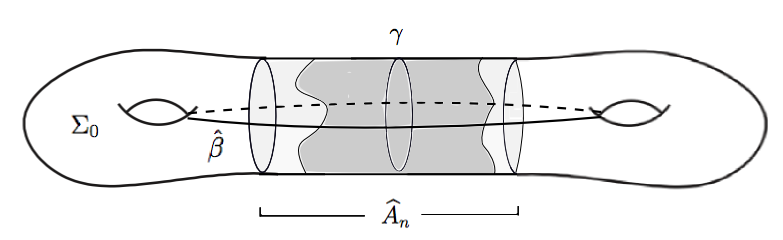}\\
 \caption{The  position of the ``truncated center" for $\mathcal{F}(\hat{q}_n)$ (shown shaded), relative to $\hat{A}_n$, is controlled by Lemma \ref{transv}.}
\end{figure}

\newpage

\begin{lem}\label{transv} For each $n\geq 1$, we have the following estimates of transverse measures:
\begin{itemize}
\item[(a)] $\tau_n(\gamma) = R\cos\alpha$.
\item[(b)] $\tau_n(\beta) \geq  2L_n\sin \alpha$, where $L_n$ is given by equation (\ref{ln}). 
\item [(c)]  $\tau_n(\beta \cap \Sigma_0) \leq T_0$, where $T_0$ is independent of $n$.
\item[(d)]  $\tau_n(\beta) \leq  2L_n\sin \alpha + T_0$.
\item [(e)]  $\tau_n(\beta \cap {A}_n) \geq 2L_n \sin \alpha - T_0$.
\end{itemize}
\end{lem}

\begin{proof}
Part (a) follows from the fact that the asserted transverse measure is that of $\partial \Sigma_n$ with respect to $F_n$. However recall that $F_n$ is the restriction of the foliation on a model cylindrical end with parameters $R$ and $\alpha$, and hence the boundary acquires the transverse measure $R\cos\alpha$ of the core curve.

Part (b) follows from the fact that by construction, the truncated center is an annulus with core curve $\gamma$, and the arc $\beta$ crosses the truncated center (see Defn. \ref{truncc})   twice, and hence the transverse measure is greater than twice the transverse measure across ${A}_n$. As in part (a), the latter computation is just that of the model cylindrical end for parameters $R,\alpha$, and gives $2L_n\sin\alpha$.

Part (c) is the most important estimate. First note that $\beta \cap \Sigma_0$ forms an arc on the fixed subsurface $\Sigma_0$ whose topological type is independent of $n$.

 The transverse measure of $\beta \cap \Sigma_0$ depends on how the foliation $\mathcal{F}(q_n)$  (that is equivalent to $F_n$) and in particular, the truncated center $A(q_n)$ meets the subsurface $\Sigma_0$.  (We need to check that the truncated center $A(q_n)$  does not move ``far" from $A_n$.)

Note that the transverse measure with respect to $\mathcal{F}(q_n)$ of a fixed embedded arc in the \textit{complement} of the truncated center is uniformly bounded. 

\begin{figure}
  \centering
  \includegraphics[scale=0.35]{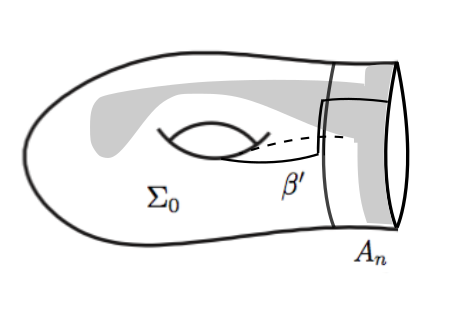}\\
 \caption{Construction of the arc $\beta^\prime$ on $\Sigma_n$  avoiding a ``tongue" of the truncated center (shown shaded).}
\end{figure}

Hence, it suffices to prove that  the truncated center for $\mathcal{F}(q_n)$  cannot venture too ``deep" in $\Sigma_0$. For this, we consider the following two cases: \\
(1) The intersection of the truncated center with $\Sigma_0$ is an annular region with $\partial \Sigma_0$ as one of the boundaries. In this case, the  modulus of such a peripheral annulus that can embed in $\Sigma_0$  is bounded above by the reciprocal of the extremal length of $\partial \Sigma_0$.  The latter is a fixed quantity (independent of $n$). \\
(2) The intersection of the truncated center with $\Sigma_0$ is a union of simply-connected regions, each bounded by an arc along $\partial \Sigma_0$, and another lying in the interior of $\Sigma_0$. Consider one component of this type. Such a  ``tongue" $\Omega$ cannot be too large in transverse measure: let $D$ denote the maximal transverse measure of any arc in that component. Then since the total transverse measure across the truncated center is $L_n\sin\alpha$, the transverse measure of some longitudinal arc 
in the part of the truncated center $\mathcal{F}(q_n) \cap A_n$
across the annulus $A_n$ will be $L_n\sin \alpha - D$.  One can then construct an arc $\beta^\prime$ on $\Sigma_n$, homotopic to $\beta$ through a homotopy fixing the endpoints on $\partial \Sigma_n$, by joining a sub-arc  in $\Sigma_0$ in the complement of the ``tongue" $\Omega$, and two longitudinal arcs across $A_n$ as above, by two arcs lying in $\partial\Sigma_0$. 
 The transverse measure of the arc away from the ``tongue" is uniformly bounded (independent of $n$), and so is that of any arc lying on $\partial \Sigma_0$ (bounded, in fact, by $R\cos \alpha$). Adding these contributions, we derive that the transverse measure of $\beta^\prime$ satisfies
\begin{equation*}
2L_n \sin\alpha \leq
 \tau_n(\beta^\prime) = 2L_n \sin \alpha  - 2D + O(1)  
 \end{equation*}
where the inequality was observed in (b). Hence we obain a uniform upper bound on $D$, as required.

Together, (1) and (2) show that only a uniformly bounded part of the truncated center can lie  in $\Sigma_0$, in the sense that any arc contained in the intersection of $\Sigma_0$ and the truncated center, has a uniform upper bound on its transverse measure.

  Thus the truncated center contributes only a uniformly bounded transverse measure to the arc $\beta \cap \Sigma_0$.
This completes the proof of (c).
 
Part (d) now follows the previous observations when one decomposes $\beta$ into arcs across the truncated center, and its complement. 

Finally, part (e) follows from (b) and (c).
 \end{proof}

\begin{cor}\label{cor5.2}  Let $\gamma$ be any arc across the cylinder $A_n$ on $\Sigma_n$.
The transverse measure of $\gamma$ satisfies 
\begin{equation*}
L_n\sin \alpha-T_0 - \frac{R\cos\alpha}{2}  \leq \tau(\gamma) \leq   L_n \sin \alpha+T_0/2
\end{equation*}
where $T_0$ is the constant (independent of $n$) in Lemma \ref{transv}.
\end{cor}
\begin{proof} The upper and lower bounds are obtained as follows:\\
(1) If the transverse measure of some arc is less than $L_n\sin \alpha-T_0 - R\cos\alpha/2$, then two oppositely-oriented copies of that arc together with an arc in $\Sigma_0 \cup \partial \Sigma_0$ of  transverse measure bounded above (See Lemma~\ref{transv}(a),(c)) by $T_0+ R\cos\alpha $ will produce an arc homotopic to $\beta$ and of transverse measure less than $2L_n \sin\alpha  - T_0$, contradicting Lemma~\ref{transv} (e). \\
(2) If the transverse measure of some arc is more than $L_n \sin \alpha+T_0/2$, then two oppositely-oriented copies of the arc together with an arc in $\Sigma_0$ produces an  arc transversely homotopic to $\beta$ and of transverse measure greater than $2L_n \sin\alpha + T_0$, contradicting Lemma \ref{transv} (d). \qed
\end{proof}

We combine the above estimate with an energy estimate in \S4 to provide the following crucial \textit{lower} bound on energy.

\begin{lem}\label{lower} For all $n$, and any lift $\widetilde{A}_n$ of the annulus $A_n$, we have 
\begin{equation*}
\bar{\mathcal{E}}(\widetilde{h}_n\vert_{\widetilde{A}_n}) \geq \frac{1}{2}{R^2}M_n - K_0 
\end{equation*}
where $M_n$ is the modulus of ${A}_n$, and $K_0>0$ is independent of $n$.

(Recall that $\bar{\mathcal{E}}(\cdot)$ is the equivariant energy, or equivalently, the energy of the restriction to a fundamental domain.)  
\end{lem}
\begin{proof}
Choose a fundamental domain of the $\mathbb{Z}$-action on $\widetilde{A}_n$ that we call $A_n$, by abuse of notation. 

From Corollary \ref{cor5.2}, we know that any longitudinal arc across ${A}_n$ must have transverse measure at least $L_n \sin\alpha - T_0 - R\cos\alpha$.

Moreover, note that any meridional circle in ${A}_n$  is homotopic to $\gamma$ and hence has transverse  measure at least $R\cos\alpha$ by  part (a) of Lemma \ref{transv}.

Then Proposition \ref{minim} applies with $C = R\cos\alpha$ and $\tau = L_n \sin \alpha - T_0$, and provides the lower bound
\begin{align*}
\bar{\mathcal{E}}(\widetilde{h}_n\vert_{\widetilde{A}_n})  &\geq \frac{1}{2}M_n \cdot \left(R^2\cos^2\alpha+ \frac{\left(L_n \sin\alpha - (T_0- R\cos\alpha/2) \right)^2}{L_n^2/R^2}\right) \\
&\geq \frac{1}{2}M_n \cdot \left(R^2\cos^2\alpha+ \frac{L_n^2\sin^2\alpha - 2(T_0 - R\cos\alpha/2)\cdot L_n\sin\alpha}{L_n^2/R^2}\right)\\
&= \frac{1}{2}M_nR^2  - {(T_0 -  R\cos\alpha/2)R\sin\alpha}
\end{align*}
which is of the required form. \end{proof}

\textit{Remark.} The upper bound in Corollary \ref{cor5.2} is used the final section of the paper.

\subsubsection{Step 2, continued:  Convergence of $\widetilde{h}_n$}

Recall that $\Sigma_n = \Sigma_0 \cup A_{n}$
where $A_{n}$ is a  conformal annulus, obtained by excising subdisks of  a chosen coordinate disk $U$ about the puncture $P$. (See the beginning of \S5.2.1)

We shall show that the sequence of harmonic maps  defined on their universal covers converges after passing to a subsequence. 

For this, we shall need the following uniform energy estimate that shall use the work of the previous subsection:

\begin{lem}[Uniform energy bound]\label{ebd} The equivariant energy $\bar{\mathcal{E}}\left(\widetilde{h}_n\vert_{\widetilde{\Sigma_0}}\right) \leq K_1$ where $K_1$  is independent of $n$. 
\end{lem}

\begin{proof}[Proof of Lemma \ref{ebd}] 
 
We use an adaptation of a technique from \cite{Wolf3}: 

We begin with the obvious equality \begin{equation*}
\bar{\mathcal{E}}(\widetilde{h}_n\vert_{\widetilde{\Sigma}_0}) + \bar{\mathcal{E}}(\widetilde{h}_n\vert_{\widetilde{A}_n})  = \bar{\mathcal{E}}(\widetilde{h}_n).
\end{equation*}

We can equivariantly build a map $\widetilde{g}$ which is a candidate for the energy-minimizing problem that $\widetilde{h}_n$ solves (see the last part of \S5.2.1), by defining $\widetilde{g}$ to equal $\widetilde{h}_0$ on $\widetilde{\Sigma}_0$, and to equal the model map  $\widetilde{m}_{R,\alpha}$ on $\widetilde{A}_n$, for each lift of $\Sigma_0$ and $A_n$.

We thus obtain:
\begin{equation}\label{enhn}
 \bar{\mathcal{E}}(\widetilde{h}_n) \leq \bar{\mathcal{E}}(\widetilde{g}) = \bar{\mathcal{E}}(\widetilde{h}_0) + \bar{\mathcal{E}}(\widetilde{m}_{R,\alpha}\vert_{\widetilde{A}_n})   =  \bar{\mathcal{E}}(\widetilde{h}_0) + \frac{1}{2}R^2\cdot \text{mod}(A_n)
 \end{equation}
where the last equality follows from Lemma \ref{ener}. 

However from Lemma  \ref{lower}  we know
\begin{equation*}
 \frac{1}{2}R^2\cdot \text{mod}(A_n)  \leq  \bar{\mathcal{E}}(\widetilde{h}_n\vert_{\widetilde{A}_n}) + K_0
\end{equation*}
where $K_0$ is independent of $n$.

By combining the three displayed inequalities, we obtain 
 \begin{equation*}
\bar{\mathcal{E}}(\widetilde{h}_n\vert_{\Sigma_0}) \leq \bar{\mathcal{E}}(\widetilde{h}_0) + K_0 
\end{equation*}
where the right hand side is independent of $n$, as required.
  \end{proof}
  
 \textit{Remark.} The same argument applies for the energy of the restriction $\widetilde{h}_n\vert_{\widetilde{\Sigma}_m}$, for any subsurface $\Sigma_{m}$ for some fixed $m\geq 1$, and hence to any compact set $K \subset \Sigma\setminus p$ since $K$ is contained in some such subsurface. \\

 Using a standard argument  (see \cite{Wolf2} for details) we then have:
 
 \begin{lem} The maps $\widetilde{h}_n$ uniformly converge to an equivariant harmonic map $\tilde{h}:\widetilde{\Sigma \setminus P} \to T$.

 \end{lem}

\begin{proof}[Sketch of the proof]
Each map $\widetilde{h}_m$ in the family $\{\widetilde{h}_m\vert_{\widetilde{\Sigma}_n}| m \geq n\}$ is harmonic, hence all are projections along the horizontal foliations of their Hopf differentials to an NPC space. Thus a disk in $\widetilde{\Sigma}_n$ gets taken by each such map into a finite subtree bounded by the image of the boundary of the disk.  
Here we notice and use that such balls contain but a finite number of zeroes, hence have image with but a finite number of vertices.
Thus, the proof of the Courant-Lebesgue Lemma together with the uniform energy bound on the compact domain $\Sigma_n$ shows that these maps $\{\widetilde{h}_m\vert_{\widetilde{\Sigma}_n}| m \geq n\}$ are equicontinuous on $\widetilde{\Sigma}_n$.
Then, we would like to use Arzela-Ascoli to provide for convergence of a subsequence of those maps on that domain.  For this, note  the target $T$  is, in general, a non-locally-compact   $\mathbb{R}$-tree, however by the work in \cite{KorSch2}  the uniform boundedness is all that is needed (see Theorem 2.1.3  and Remark 2.1.4 of that paper).

We provide a sketch of the argument in Lemma 3.4 of \cite{Wolf2} proving the images are uniformly bounded: 

Consider two simple closed curves $B$ and $C$  in $\Sigma_n$ that intersect,  and are transverse to the foliation. Let $\gamma_B$ and $\gamma_C$ be the corresponding elements of the fundamental group that act on the universal cover and tree $T$, preserving bi-infinite axes for each.  In particular, the axes intersect in the image of a fundamental domain $\mathcal{D}$. The key observation is that if $w\in T$, then using the fact that $T$ is a tree  and the action of $\pi_1(\Sigma_n)$ on $T$ is by isometries we find
\begin{center}
$d_T(w,\gamma_B w) > 2d_T(w, \text{axis of }\gamma_B)$
\end{center}
(and similarly for $\gamma_C$). 
By the equivariance of the action of $\pi_1(\Sigma_n)$, we have $\widetilde{h_n}(\gamma_Bz) = \gamma_B\widetilde{h_n}(z)$ and $\widetilde{h_n}(\gamma_Cz) = \gamma_C\widetilde{h_n}(z)$.  Using the equicontinuity bound and the fact that  $\mathcal{D} \cup \gamma_B\mathcal{D} \cup \gamma_C\mathcal{D}$ is of bounded diameter,
 we then have
 \begin{center}
 $d_T(\widetilde{h_n}(z), \widetilde{h_n}(\gamma_Bz))\leq C_1$ and  $d_T(\widetilde{h_n}(z), \widetilde{h_n}(\gamma_Cz))\leq C_1$ 
 \end{center}
 for any $z\in \mathcal{D}$, where $C_1$ is independent of $n$.
  
  From these facts we can derive that both 
  \begin{center}
  $d_T(\widetilde{h_n}(z), \text{axis of } \gamma_B)\leq C_1/2\text{   }$ and $\text{   } d_T(\widetilde{h_n}(z), \text{axis of } \gamma_C)\leq C_1/2$ 
  \end{center}
 for any $z\in \mathcal{D}$. Since these axes $\gamma_B$ and $\gamma_C$ diverge in $T$, the image $\widetilde{h_n}(z)$ lies in a set of uniformly bounded diameter, as required.

As a consequence, Arzela-Ascoli applies, and we obtain a convergent subsequence of the family. A diagonal argument applied to the exhaustion $\{\Sigma_n\}$ of $\Sigma$ then finishes the argument. 
\end{proof}

\subsubsection{Step 3: The Hopf differential has the required foliation} 

Once we have the limiting harmonic map $\tilde{h}:\widetilde{\Sigma \setminus P} \to T$, 
it remains to check that its Hopf differential is the desired one, namely,  that it  descends to $\Sigma \setminus P$ to yield a quadratic differential  that  has:

\begin{itemize}
\item[(a)] poles of order two at the punctures,   
\item[(b)] the required residues at the poles, and 
\item[(c)] an induced measured foliation that is measure-equivalent to $F$.
\end{itemize}

We provide the proof  of these in this subsection. \\

Recall that we have a sequence of equivariant harmonic maps $\widetilde{h_n}:\widetilde{\Sigma_n} \to T$ where $\widetilde{h_n}\to \widetilde{h}$ uniformly on compact sets.

\subsubsection*{Proof of (a)}

Since (see Lemma~\ref{lower}) the equivariant energies of the harmonic map $\bar{\mathcal{E}}(\widetilde{h}_n) \to \infty$ as $n\to \infty$ in the lifts of any neighborhood of each puncture of $\Sigma$, the Hopf differential of $h$ has  poles of order greater or equal to two at any puncture.

So at any fixed puncture, the Hopf differential has the local form $z^{-k}Q(z)dz^2$ for some $k\geq 2$, where $Q(z)$ is a bounded holomorphic function. 

Assume $k>2$.  Then if we uniformize $A_{n}$ to a round annulus of outer radius $1$ and inner radius $0<r<1$, we have the following computation of  energy:
\begin{equation*}
\bar{\mathcal{E}}(\widetilde{h}_n\vert_{\widetilde{A}_{n}}) =\frac{1}{2} \int\limits_0^{2\pi} \int\limits_r^1 \lvert z^{-k}Q(z)\rvert rdrd\theta  \asymp \frac{1}{r^{k-2}}.
\end{equation*}
Now note that $r^{-1} = e^{2\pi\text{mod}(A_{n})}$, so that $\bar{\mathcal{E}}(\widetilde{h}_n\vert_{\widetilde{A}_{n}}) \asymp e^{2\pi(k-2)\text{mod}(A_{n})}$.\\
Yet estimate (\ref{enhn}) asserts that the equivariant energy of the restriction to $\widetilde{A}_n$  is  at most linear in the modulus of $A_n$, so it is not possible that $k>2$.
 (Indeed, for $k=2$ the same calculation above yields an expression that is linear in $\ln r$.)

\subsubsection*{Proof of (b)}
We first refine the previous estimate and show that the distance of each approximating $\widetilde{h}_n$ from the lift of the model map $\widetilde{m}_{R,\alpha}$ is uniformly bounded on their restrictions  on the  strip $\widetilde{A}_n = [0,L_n] \times \mathbb{R}$ (a lift of the annulus $A_{n}$).   Recall that $m_{R,\alpha}$ was the model map corresponding to the desired residue at the pole.

\begin{lem}\label{modelBound}
The distance function  $d_T(\widetilde{h}_n\vert_{\widetilde{A}_n}, \widetilde{m_{R,\alpha}}\vert_{\widetilde{A}_n}) < d_0$ where $d_0$ is  independent of $n$.
\end{lem}
\begin{proof}
We first consider the  above distance function 
on the lifts of $\partial A_n = \partial\Sigma_0 \sqcup \partial\Sigma_n$.

First, by construction, the map $\widetilde{h_n}$ 
coincides with the  lift of the model map $\tilde{m}_{R,\alpha}$ on the lift of $\partial \Sigma_n$, so our attention turns to the distance between the two maps on the lift of $\partial\Sigma_0$.

To that end, choose a point, say $p_0$, on the boundary $\partial\Sigma_0$, and connect it by an arc $\gamma$ to a point on the 
boundary $\partial\Sigma_n$.  
Corollary~\ref{cor5.2} then asserts that the transverse measure of $\gamma$ with respect to the foliation $\mathcal{F}(q_n)$ is $L_n \sin \alpha$, up to an additive distortion $d_0 := T_0 + R\cos\alpha/2$  which is itself uniformly bounded (independent of $n$).  
Thus the image $\widetilde{h_n}(\tilde{\gamma})$ of the lift $\tilde{\gamma}$ of this arc under $\widetilde{h_n}$ has length $L_n \sin \alpha$, again up to an additive distortion of $d_0$.  
Since the length of the image of $\tilde{\gamma}$ by $\widetilde{m_{R,\alpha}}$ is exactly $L_n \sin \alpha$, so we obtain $d_T(\widetilde{h}_n(p_0), \widetilde{m_{R,\alpha}}(p_0) < d_0$.  As the choice of $p_0$ on the lift of $\partial \Sigma_0$ was arbitrary, we can conclude that the distance function is uniformly bounded on the lift of that boundary component of $A_n$ as well.

  However the distance function $d_T$  is subharmonic, as well as invariant under deck transformations. Hence the Maximum Principle applies to the function $d_T$ on the quotient $A_n$, and we find that $d_T$ is bounded in the interior of the strip by $d_0$, which is independent of $n$.
  \end{proof}

We have already proved the uniform convergence $\widetilde{h}_n \to \widetilde{h}$, hence the images of the lifts of $\partial \Sigma_0$ under $\widetilde{h}_n$ are uniformly bounded. 
Hence the limiting harmonic map $h$ is asymptotic to (\textit{i.e.} bounded distance from) the model map $m_{R,\alpha}$  on the punctured disk $U \setminus P = \bigcup\limits_n A_n$, which is a topological  end of the surface $\Sigma \setminus P$.

By  Lemma \ref{dist}, the argument $\alpha$ of the residue $a$ at the pole agrees with that of the model map $m_{R,\alpha}$, and so does its real part which is determined by the transverse measure of the foliation $F$ (see equation (\ref{trans}) in \S1). 

When this transverse measure is positive, the above real part is $R\cos\alpha$. In this case, knowing $R\cos\alpha$ and $\alpha$ determines $R$, the modulus $\lvert a \rvert$  of the residue as well. 
This shows that the map $h$ has the required residue at the pole, in this case.

When the transverse measure is zero, we are in the case of a closed center at the pole, and the only parameter to determine is $R$. The diameter in $T$ of the image of  a lift of $A_n$ under the model map $\tilde{m}_{R,\pm \frac{\pi}{2}}$ is $L_n$, which is related to the modulus of the annulus by  $L_n= R\text{mod}(A_n)$  (see (\ref{ln}). Since $\text{mod}(A_n) \to \infty$ as $ n\to \infty$,  by Lemma \ref{modelBound} we have
\begin{equation*}
\frac{ \text{diam}\left(\widetilde{h}_n ({\widetilde{A}_n})\right)}{\text{mod}(A_n)}  = \frac{L_n + O(1)}{ \text{mod}(A_n)} \to R  \text{ as } n \to \infty 
\end{equation*}
and this specifies $R$ uniquely.

Hence the  map $h$ has the required residue at the pole, concluding the proof of (b).

\subsubsection*{Proof of (c)}
Recall that by construction, each map $\widetilde{h}_n$ collapses along a foliation that is equivalent to $\widetilde{F}\vert_{\widetilde{\Sigma}_n}$.  Also, note that like the Dehn-Thurston parameters, the ``intersection numbers" of a measured foliation with \textit{all} simple closed curves  (that is, the transverse measures of the curves) also serve to characterize the measured foliation   - see, for example, \cite{FLP}.

Let $\widetilde{F}^\prime$ be the collapsing foliation of $\widetilde{h}$, the uniform limit of $\widetilde{h}_n$ as $n\to \infty$. By the equivariance of $\widetilde{h}$, this descends to a measured foliation $F^\prime$ on the surface $\Sigma \setminus P$.  Similarly, the collapsing foliation of each $\widetilde{h}_n$ descends to define a  measured foliation on $\Sigma_n$, that we call $F^\prime_n$.

 To verify that the collapsing foliation $F^\prime$  is equivalent to $F$, one needs to check that the intersection number of $F^\prime$  with any simple closed curve, and the loop around the puncture, agrees with that of $F$.

First, the intersection numbers agree for the loop linking the puncture, since the parameters $R, \alpha$ of the model map were chosen so that the transverse measure of the loop is $R\cos\alpha$.  

Second,   any given simple closed curve $\gamma$ on $\Sigma$  lies in a  compact set $K \subset \Sigma_m$ for $m$ sufficiently large, and hence its intersection number with $F^\prime_n$  coincides with  $i(\gamma,F)$ for all $n\geq m$. However, since the convergence $\widetilde{h}_n\to \widetilde{h}$ is uniform on a lift of the compact set $K$, the collapsing foliations of $\widetilde{h}_n$ converge to $\widetilde{F}^\prime$ on $K$ (as is standard, the $C^0$-convergence of the sequence of holomorphic Hopf differentials can be bootstrapped to a $C^1$-convergence by the Cauchy integral formula), and hence the intersection numbers with $\gamma$ converge to  $i(\gamma, F^\prime)$.  Thus this limiting intersection number $i(\gamma, F^\prime)$ equals $i(\gamma, F)$, concluding the argument.

\bibliographystyle{amsalpha}
\bibliography{qdref3}

\def\cprime{$'$}
\providecommand{\bysame}{\leavevmode\hbox to3em{\hrulefill}\thinspace}
\providecommand{\MR}{\relax\ifhmode\unskip\space\fi MR }
\providecommand{\MRhref}[2]{%
  \href{http://www.ams.org/mathscinet-getitem?mr=#1}{#2}
}
\providecommand{\href}[2]{#2}
\begin{thebibliography}{ALPS16}

\bibitem[ALPS16]{ALPS}
D.~Alessandrini, L.~Liu, A.~Papadopoulos, and W.~Su, \emph{The horofunction
  compactification of {T}eichm\"uller spaces of surfaces with boundary},
  Topology Appl. \textbf{208} (2016), 160--191.

\bibitem[DDW00]{DaDoWen}
G.~Daskalopoulos, S.~Dostoglou, and R.~Wentworth, \emph{On the
  {M}organ-{S}halen compactification of the {${\rm SL}(2,{\bf C})$} character
  varieties of surface groups}, Duke Math. J. \textbf{101} (2000), no.~2,
  189--207.

\bibitem[DW07]{DaWen}
Georgios~D. Daskalopoulos and Richard~A. Wentworth, \emph{Harmonic maps and
  {T}eichm\"uller theory}, Handbook of {T}eichm\"uller theory. {V}ol. {I}, IRMA
  Lect. Math. Theor. Phys., vol.~11, Eur. Math. Soc., Z\"urich, 2007,
  pp.~33--109.

\bibitem[FLP12]{FLP}
Albert Fathi, Fran{\c{c}}ois Laudenbach, and Valentin Po{\'e}naru,
  \emph{Thurston's work on surfaces}, Mathematical Notes, vol.~48, Princeton
  University Press, Princeton, NJ, 2012, Translated from the 1979 French
  original by Djun M. Kim and Dan Margalit.

\bibitem[GW]{GW2}
Subhojoy Gupta and Michael Wolf, \emph{Meromorphic quadratic differentials and
  measured foliations on a {R}iemann surface}, {\it (In preparation)}.

\bibitem[GW16]{GW1}
\bysame, \emph{Quadratic differentials, half-plane structures, and harmonic
  maps to trees}, Comment. Math. Helv. \textbf{91} (2016), no.~2, 317--356.

\bibitem[HM79]{HubbMas}
John Hubbard and Howard Masur, \emph{Quadratic differentials and foliations},
  Acta Math. \textbf{142} (1979), no.~3-4.

\bibitem[Jos91]{Jost2}
J{\"u}rgen Jost, \emph{Two-dimensional geometric variational problems}, Pure
  and Applied Mathematics (New York), John Wiley \& Sons, Ltd., Chichester,
  1991, A Wiley-Interscience Publication.

\bibitem[Kap09]{Kap}
Michael Kapovich, \emph{Hyperbolic manifolds and discrete groups}, Modern
  Birkh\"auser Classics, Birkh\"auser Boston Inc., Boston, MA, 2009.

\bibitem[KS93]{KorSch1}
Nicholas~J. Korevaar and Richard~M. Schoen, \emph{Sobolev spaces and harmonic
  maps for metric space targets}, Comm. Anal. Geom. \textbf{1} (1993), no.~3-4.

\bibitem[KS97]{KorSch2}
\bysame, \emph{Global existence theorems for harmonic maps to non-locally
  compact spaces}, Comm. Anal. Geom. \textbf{5} (1997), no.~2, 333--387.

\bibitem[Liu08]{Liu}
Jinsong Liu, \emph{Jenkins-{S}trebel differentials with poles}, Comment. Math.
  Helv. \textbf{83} (2008), no.~1, 211--240.

\bibitem[Mes02]{Mese}
Chikako Mese, \emph{Uniqueness theorems for harmonic maps into metric spaces},
  Commun. Contemp. Math. \textbf{4} (2002), no.~4, 725--750.

\bibitem[MP98]{MulPenk}
M.~Mulase and M.~Penkava, \emph{Ribbon graphs, quadratic differentials on
  {R}iemann surfaces, and algebraic curves defined over {$\overline{\bold
  Q}$}}, Asian J. Math. \textbf{2} (1998), no.~4, Mikio Sato: a great Japanese
  mathematician of the twentieth century.

\bibitem[PH92]{PenHar}
R.~C. Penner and J.~L. Harer, \emph{Combinatorics of train tracks}, Annals of
  Mathematics Studies, vol. 125, Princeton University Press, Princeton, NJ,
  1992.

\bibitem[Sch]{Schoen}
Richard~M. Schoen, \emph{Analytic aspects of the harmonic map problem}, Seminar
  on nonlinear partial differential equations ({B}erkeley, {C}alif., 1983),
  Math. Sci. Res. Inst. Publ., vol.~2, pp.~321--358.

\bibitem[Sch84]{Sch}
\bysame, \emph{Analytic aspects of the harmonic map problem}, Seminar on
  nonlinear partial differential equations ({B}erkeley, {C}alif., 1983), Math.
  Sci. Res. Inst. Publ., vol.~2, Springer, New York, 1984, pp.~321--358.

\bibitem[Str84]{Streb}
Kurt Strebel, \emph{Quadratic differentials}, Ergebnisse der Mathematik und
  ihrer Grenzgebiete (3) [Results in Mathematics and Related Areas (3)],
  vol.~5, Springer-Verlag, 1984.

\bibitem[Sun03]{Sun}
Xiaofeng Sun, \emph{Regularity of harmonic maps to trees}, Amer. J. Math.
  \textbf{125} (2003), no.~4, 737--771.

\bibitem[Thu]{DehnThur}
Dylan Thurston, \emph{On geometric intersection of curves in surfaces, draft at
  \url{http://pages.iu.edu/~dpthurst/DehnCoordinates.pdf}}.

\bibitem[Wol89]{Wolf0}
Michael Wolf, \emph{The {T}eichm\"uller theory of harmonic maps}, J.
  Differential Geom. \textbf{29} (1989), no.~2, 449--479.

\bibitem[Wol91]{Wolf3}
\bysame, \emph{Infinite energy harmonic maps and degeneration of hyperbolic
  surfaces in moduli space}, J. Differential Geom. \textbf{33} (1991), no.~2,
  487--539.

\bibitem[Wol95]{WolfT}
\bysame, \emph{Harmonic maps from surfaces to {$\bold R$}-trees}, Math. Z.
  \textbf{218} (1995), no.~4, 577--593.

\bibitem[Wol96]{Wolf2}
\bysame, \emph{On realizing measured foliations via quadratic differentials of
  harmonic maps to {$\bold R$}-trees}, J. Anal. Math. \textbf{68} (1996),
  107--120.

\end{thebibliography}

\end{document}